\documentclass[11pt]{article}

\usepackage[margin=1in]{geometry}
\usepackage{amsmath,amssymb,amsthm,mathtools}
\usepackage{enumitem}
\usepackage[hidelinks]{hyperref}
\usepackage{microtype}
\usepackage{tikz}

\hypersetup{
  pdftitle={Extremal Graphs for the Energy--Independence Number Inequality},
  pdfauthor={Seyed Ahmad Mojallal},
  pdfsubject={Spectral graph theory and extremal graph energy},
  pdfkeywords={graph energy, independence number, vertex-cover number,
    extremal graph, equality characterization}
}

\newtheorem{theorem}{Theorem}[section]
\newtheorem{lemma}[theorem]{Lemma}
\newtheorem{proposition}[theorem]{Proposition}
\newtheorem{corollary}[theorem]{Corollary}
\theoremstyle{definition}
\newtheorem{definition}[theorem]{Definition}
\theoremstyle{remark}
\newtheorem{remark}[theorem]{Remark}

\newcommand{\E}{\mathcal{E}}
\newcommand{\one}{\mathbf{1}}
\newcommand{\Spec}{\operatorname{Spec}}
\newcommand{\tr}{\operatorname{tr}}

\title{Extremal Graphs for the Energy--Independence Number Inequality}

\author{
Seyed Ahmad Mojallal%
\thanks{Department of Mathematics, Simon Fraser University,
Burnaby, BC V5A 1S6, Canada.
Email: \texttt{ahmad\_mojalal@yahoo.com}.}
}

\date{}

\begin{document}

\maketitle

\begin{abstract}
For a graph \(G\) of order \(n\), let \(\mathcal E(G)\) denote its adjacency
energy and let \(\alpha(G)\) denote its independence number.  A recent theorem
of Kumar and Pragada states that
\[
   \mathcal E(G)\ge 2\bigl(n-\alpha(G)\bigr).
\]
We determine all graphs attaining equality.  More precisely, equality holds
if and only if every connected component of \(G\) is an isolated vertex, a
balanced complete multipartite graph, or a graph obtained by taking the disjoint union of
\(K_{a,\ldots,a}\) and \(K_{b,\ldots,b}\), with the same number \(r\ge3\) of
parts, and then completely joining corresponding parts.
\end{abstract}



\section{Introduction}

Throughout the paper, all graphs are finite, simple, and undirected.  For a
graph $G$, we write $V(G)$ and $E(G)$ for its vertex and edge sets,
$n=|V(G)|$ for its order, and $\alpha(G)$ for its independence number.  For
$v\in V(G)$, let $N(v)$ and $N[v]:=N(v)\cup\{v\}$ denote its open and closed
neighborhoods, respectively.  The vertex-cover number is
$\tau(G)=n-\alpha(G)$.
A complete multipartite graph with part sizes $a_1,\ldots,a_r$ is denoted by
$K_{a_1,\ldots,a_r}$.  It is called \emph{balanced} when all part sizes are
equal, in which case we write $K_{a,a,\ldots,a}$.  A graph is
\emph{well-covered} if every maximal independent set has cardinality
$\alpha(G)$.

Let \(A(G)\) be the adjacency matrix of \(G\), with eigenvalues
\(\lambda_1,\ldots,\lambda_n\). The adjacency energy of \(G\) is defined by
\[
   \E(G):=\sum_{i=1}^n |\lambda_i|.
\]
This spectral invariant was introduced by Gutman in 1978, motivated by its
connection with the total \(\pi\)-electron energy of conjugated hydrocarbon
molecules; see \cite{Gutman1978}.

The inequality studied in this paper has a history spanning several decades.
Using the automated conjecture-making program Graffiti, Fajtlowicz proposed
in the 1980s that
\begin{equation}\label{eq:KP-main}
   \E(G)\ge 2\bigl(n-\alpha(G)\bigr)=2\tau(G)
\end{equation}
for every graph \(G\). An early account of the Graffiti project appeared in
\cite{Fajtlowicz1987}, and the conjecture was subsequently recorded as
Conjecture~543 in the survey of Aouchiche and Hansen
\cite[Table~6]{AouchicheHansen2010}.

Several partial results preceded the complete proof. 
Wang and Ma proved the lower bound
\[
   \E(G)\ge 2\bigl(\tau(G)-c(G)\bigr),
\]
where \(c(G)\) denotes the number of odd cycles in their formulation
\cite{WangMa2017}. Liu and Ning later included the conjecture in their survey
of unsolved problems in spectral graph theory and reported computational
verification for graphs of order at most ten \cite{LiuNing2023}.

Further progress was obtained in 2025. Akbari, K\"u\c{c}\"uk\c{c}if\c{c}i,
Saveh, and Yaz{\i}c{\i} established stronger vertex-cover bounds and verified
\eqref{eq:KP-main} for several graph classes
\cite{AkbariEtAl2025}. Samanta proved the inequality for additional families
of graphs \cite{Samanta2025}. Abiad, Coutinho, Juliano, and Reijnders
introduced a semidefinite-programming formulation of graph energy and derived
several substantial partial results toward the conjecture
\cite{AbiadEtAl2025}.

Finally, Kumar and Pragada proved \eqref{eq:KP-main} for every graph in 2026
\cite{KumarPragada}.  The purpose of the present paper is to determine all
graphs for which equality holds.

Equality has previously been studied for related energy bounds.  Wang and Ma
characterized equality in their lower bound involving the number of odd
cycles \cite{WangMa2017}, while Chen and Liu characterized the graphs whose
energy is twice their matching number \cite{ChenLiu2021}.  These results do
not determine equality in \eqref{eq:KP-main}.  Abiad, Coutinho, Juliano, and
Reijnders observed that the known equality examples for
\eqref{eq:KP-main} include balanced complete multipartite graphs and the
Cartesian products $K_2\square K_r$ \cite{AbiadEtAl2025}.  To the best of our
knowledge, no complete characterization of the equality graphs has previously
been obtained.

We next define the second extremal family appearing in our characterization.
It extends the previously known examples $K_2\square K_r$.

\begin{definition}\label{def:Hrab}
Let $r\ge 2$ and let $a,b\ge 1$ be integers.  The graph $H_r(a,b)$
has a vertex partition
\[
  X_1\dot\cup\cdots\dot\cup X_r\dot\cup
  Y_1\dot\cup\cdots\dot\cup Y_r,
\]
where $|X_i|=a$ and $|Y_i|=b$ for every $i$.  Each $X_i$ and each
$Y_i$ is an independent set, and the adjacencies are defined as follows:
\begin{enumerate}[label=\textup{(\roman*)}]
\item $X_i$ is completely joined to $X_j$ whenever $i\ne j$;
\item $Y_i$ is completely joined to $Y_j$ whenever $i\ne j$;
\item $X_i$ is completely joined to $Y_i$ for every $i$;
\item there are no edges between $X_i$ and $Y_j$ whenever $i\ne j$.
\end{enumerate}
\end{definition}

\begin{remark}\label{rem:Hrab-blowup}
We have $H_2(a,b)\cong K_{a+b,a+b}$.
Thus, the members of this family that are not already balanced complete
multipartite graphs occur when $r\ge 3$.

For graphs $G_1$ and $G_2$, their Cartesian product
$G_1\square G_2$ has vertex set $V(G_1)\times V(G_2)$, where
$(u,x)$ and $(v,y)$ are adjacent if either $u=v$ and
$xy\in E(G_2)$, or $x=y$ and $uv\in E(G_1)$.

The graph $H_r(a,b)$ may be viewed as a two-weight independent blow-up
of $K_2\square K_r$.  More precisely, each vertex in one copy of $K_r$
is replaced by an independent set of size $a$, each vertex in the other
copy is replaced by an independent set of size $b$, and every edge is
replaced by all possible edges between the corresponding independent
sets.  In particular,
\[
   H_r(1,1)\cong K_2\square K_r.
\]
When $a=b$, the graph $H_r(a,a)$ is the uniform independent blow-up of
$K_2\square K_r$, obtained by replacing every vertex with an independent
set of size $a$.  For $a\ne b$, the two copies of $K_r$ are replaced
using different blow-up factors.
\end{remark}

\begin{theorem}[Equality characterization]\label{thm:main}
Let $G$ be a graph of order $n$.  Then
\[
   \E(G)=2\bigl(n-\alpha(G)\bigr)
\]
if and only if every connected component of $G$ is one of the following:
\begin{enumerate}[label=\textup{(\arabic*)}]
\item an isolated vertex;
\item a balanced complete multipartite graph
      $K_{a,a,\ldots,a}$ with at least two parts;
\item a graph $H_r(a,b)$ with $r\ge 3$ and $a,b\ge 1$.
\end{enumerate}
\end{theorem}

The proof combines spectral, geometric, and combinatorial ideas.  Equality in
the neighborhood-deletion inequality first propagates to every
anti-neighborhood and implies that every equality graph is well-covered.
Equality in the local positive-semidefinite estimates then shows that every
edge satisfies one of two precise relations among the associated Gram vectors.
These relations lead naturally to a partition of the vertex set into
$P$-classes and cells.  The adjacencies among the cells are encoded by a
weighted bipartite incidence graph: the graph whose vertices are the cells is
the line graph of this bipartite graph, with cell cardinalities serving as
weights.

The final step uses the standard equality between the maximum weight of a
matching in a bipartite graph and the maximum weight of a fractional matching.
A uniform fractional matching is optimal, and comparison with an optimal
solution of the associated dual problem shows that every $P$-class contains
exactly $n/\alpha(G)$ cells, all of the same cardinality.  These conclusions
imply that there are at most two $P$-classes and yield the two connected
equality families.

The paper is organized as follows.  Section~2 recalls the spectral
decomposition and proves that equality propagates to anti-neighborhoods.
Section~3 derives rigidity of the absolute adjacency matrix.  Section~4 treats
the boundary case $\alpha(G)=n/2$.  Section~5 handles the strict case
$\alpha(G)<n/2$: it establishes the two possible edge types, introduces $P$-classes
and cells, constructs the weighted cell incidence graph, and applies the
weighted matching linear program to show that there are at most two
$P$-classes.  Section~6 identifies the resulting connected graphs.
Section~7 verifies that the listed families attain equality, using the
standard spectrum of balanced complete multipartite graphs and a direct
calculation for $H_r(a,b)$.  Section~8 proves the main theorem, including the
disconnected case, and Section~9 records several consequences.

\section{Preliminaries and equality propagation}
In this section, we collect the spectral notation and the basic identities
needed throughout the proof.  We then revisit the neighborhood-deletion
argument of Kumar and Pragada and determine its first consequences in the
equality case.  In particular, we show that equality is inherited by every
anti-neighborhood and deduce that every equality graph is well-covered.

For a real symmetric matrix $M$, the notation $M\succeq0$ means that $M$ is
positive semidefinite.  Equivalently, $M$ is a Gram matrix: there exist
Euclidean vectors $z_1,\ldots,z_n$ such that
$M_{ij}=\langle z_i,z_j\rangle$.  We write $I_k$ and $J_k$ for the identity
and all-ones matrices of order $k$, respectively, and $\one$ for an all-ones
vector of the dimension indicated by the context.

Fix a graph $G$, and let $A=A(G)$.  Write
\begin{equation}\label{eq:spectral-parts}
   A=P-Q,
   \qquad P=A_+\succeq 0,
   \qquad Q=A_-\succeq 0,
   \qquad PQ=0,
\end{equation}
where $A_+$ and $A_-$ are the positive and negative spectral parts of $A$.
Set
\[
   B:=P+Q=|A|.
\]
Then $B^2=A^2$.  Since $A$ has zero diagonal,
\begin{equation}\label{eq:diagPQ}
   P_{vv}=Q_{vv}=\frac12 B_{vv}
   \qquad(v\in V(G)).
\end{equation}
Moreover,
\begin{equation}\label{eq:traceP}
   \tr P=\tr Q=\frac12\E(G).
\end{equation}

For $v\in V(G)$, write
\[
   G_v:=G-N[v].
\]
The key inequality of Kumar and Pragada is
\begin{equation}\label{eq:neighborhood-deletion}
   4|E(G)|+\sum_{v\in V(G)}\E(G_v)
   \le |V(G)|\E(G).
\end{equation}
See \cite[Lemma 2.2]{KumarPragada}.

\begin{proposition}[Equality propagates to anti-neighborhoods]
\label{prop:propagation}
Suppose that
\[
   \E(G)=2\bigl(n-\alpha(G)\bigr).
\]
Then, for every $v\in V(G)$,
\begin{align}
   \alpha(G_v)&=\alpha(G)-1,\label{eq:alpha-propagation}\\
   \E(G_v)&=2\bigl(|V(G_v)|-\alpha(G_v)\bigr).
   \label{eq:energy-propagation}
\end{align}
In particular, every equality graph is well-covered.
\end{proposition}

\begin{proof}
Put $\alpha=\alpha(G)$, $m=|E(G)|$, and $d(v)=\deg_G(v)$.  Since every
independent set in $G_v$ can be enlarged by adding $v$, we have
\[
   \alpha(G_v)\le \alpha-1.
\]
Also, $|V(G_v)|=n-1-d(v)$.  Applying \eqref{eq:KP-main} to $G_v$ gives
\begin{align*}
 \E(G_v)
 &\ge 2\bigl(n-1-d(v)-\alpha(G_v)\bigr)\\
 &\ge 2\bigl(n-\alpha-d(v)\bigr).
\end{align*}
Summing over all vertices and using $\sum_v d(v)=2m$, we obtain
\begin{equation}\label{eq:lower-sum-antineighborhood}
   \sum_v \E(G_v)
   \ge 2n(n-\alpha)-4m.
\end{equation}
On the other hand, \eqref{eq:neighborhood-deletion} and the equality hypothesis
imply
\[
   \sum_v\E(G_v)
   \le n\E(G)-4m
   =2n(n-\alpha)-4m.
\]
Hence equality holds in \eqref{eq:lower-sum-antineighborhood}.  For each
vertex $v$, both gaps in the two-step lower estimate are nonnegative.  Since
the sum of all these gaps is zero, both inequalities are equalities for every
$v$.  This gives \eqref{eq:alpha-propagation} and
\eqref{eq:energy-propagation}.

To prove that $G$ is well-covered, let $S=\{v_1,\ldots,v_t\}$ be an independent
set.  Starting with $G$, successively pass to the anti-neighborhood of $v_i$ in
the graph remaining after the previous deletions; since $S$ is independent,
each $v_i$ is still present when it is selected.  Equation
\eqref{eq:energy-propagation} ensures that every intermediate graph is again
an equality graph, while \eqref{eq:alpha-propagation} shows that its
independence number drops by exactly one at each step.  Consequently the
remaining graph has independence number $\alpha(G)-t$, so $S$ extends to an
independent set of size $\alpha(G)$.
\end{proof}

\section{Rigidity of the absolute adjacency matrix}

We next track equality in the local $2\times2$ positive-semidefinite estimate
used in the proof of \eqref{eq:neighborhood-deletion}.

\begin{lemma}[Edge rigidity]\label{lem:edge-rigidity}
Let $G$ be a connected nontrivial graph satisfying equality in
\eqref{eq:KP-main}.  For every edge $uv\in E(G)$,
\begin{equation}\label{eq:edge-rigidity}
   B_{uu}=B_{vv},
   \qquad
   |B_{uv}|=B_{uu}-1.
\end{equation}
\end{lemma}

\begin{proof}
In the notation of \cite[Claim 2.2]{KumarPragada}, for an edge $uv$ put
\[
   x=B_{uu},\qquad y=B_{vv},\qquad z=B_{uv}.
\]
The $2\times2$ principal submatrices of $2P=B+A$ and $2Q=B-A$ indexed by
$u,v$ are positive semidefinite.  Since $A_{uv}=1$, their determinants give
\[
   xy\ge (z+1)^2,
   \qquad
   xy\ge (z-1)^2.
\]
Thus
\begin{equation}\label{eq:sqrtxy}
   \sqrt{xy}\ge 1+|z|.
\end{equation}
The proof of \eqref{eq:neighborhood-deletion} shows that the quantity
\begin{equation}\label{eq:Fuv}
 F_{uv}:=
 \frac{(x-1)^2-z^2}{x}
 +\frac{(y-1)^2-z^2}{y}
\end{equation}
is nonnegative for every edge.  More precisely, if $P_v$ denotes the
Schur-complement matrix used in \cite[Claim 2.1]{KumarPragada}, then the
proof gives the chain
\[
  \sum_v \E(G_v)\le 2\sum_v\tr(P_v)\le n\E(G)-4m,
\]
and the difference in the second inequality is $\sum_{uv\in E(G)}F_{uv}$.
By Proposition~\ref{prop:propagation}, the two ends of this chain are equal.
Hence both inequalities are equalities and
$\sum_{uv\in E(G)}F_{uv}=0$.  Since every $F_{uv}$ is nonnegative,
$F_{uv}=0$ for every edge.

Now
\begin{align*}
F_{uv}
&=\frac{x+y}{xy}(xy+1-z^2)-4\\
&\ge \frac{2}{\sqrt{xy}}(xy+1-z^2)-4\\
&=\frac{2}{\sqrt{xy}}\bigl((\sqrt{xy}-1)^2-z^2\bigr)\\
&\ge 0.
\end{align*}
The factor $xy+1-z^2$ is positive under \eqref{eq:sqrtxy}.  Hence equality in
the AM--GM step forces $x=y$.  Equality in the last step gives
\[
   (\sqrt{xy}-1)^2=z^2.
\]
Together with \eqref{eq:sqrtxy}, this yields
$\sqrt{xy}=1+|z|$.  Since $x=y$, we obtain
$x=y=1+|z|$, which is \eqref{eq:edge-rigidity}.
\end{proof}

\begin{corollary}\label{cor:constant-diagonal}
Let $G$ be connected, nontrivial, and satisfy equality in
\eqref{eq:KP-main}.  Put
\[
   \alpha:=\alpha(G),
   \qquad
   \tau:=n-\alpha,
   \qquad
   d:=\frac{\tau}{n},
   \qquad
   e:=\frac{\alpha}{n}.
\]
Then $d+e=1$ and
\begin{equation}\label{eq:constant-diagonal}
   P_{vv}=Q_{vv}=d,
   \qquad
   B_{vv}=2d
   \qquad(v\in V(G)).
\end{equation}
For every edge $uv$,
\begin{equation}\label{eq:Bedge}
   B_{uv}\in\{d-e,e-d\}.
\end{equation}
In particular, $\alpha\le n/2$.
\end{corollary}

\begin{proof}
By Lemma~\ref{lem:edge-rigidity}, $B_{uu}=B_{vv}$ whenever $uv$ is an edge.
Connectedness makes the diagonal of $B$ constant.  Since
\[
   \tr B=\E(G)=2\tau,
\]
we obtain $B_{vv}=2\tau/n=2d$.  Equation \eqref{eq:diagPQ} gives
\eqref{eq:constant-diagonal}.  Lemma~\ref{lem:edge-rigidity} now gives
\[
   |B_{uv}|=2d-1=d-e,
\]
which is nonnegative.  Thus $d\ge e$, equivalently $\alpha\le n/2$, and
\eqref{eq:Bedge} follows.
\end{proof}

\section{The boundary case \texorpdfstring{$\alpha=n/2$}{alpha = n/2}}

\begin{proposition}\label{prop:half}
Let $G$ be connected and nontrivial.  If
\[
   \E(G)=2\bigl(n-\alpha(G)\bigr)
   \qquad\text{and}\qquad
   \alpha(G)=\frac n2,
\]
then $G\cong K_{n/2,n/2}$.
\end{proposition}

\begin{proof}
Here $d=e=1/2$.  By \eqref{eq:Bedge}, $B_{uv}=0$ on every edge.  Therefore,
for every edge $uv$,
\[
   P_{uv}=\frac{B_{uv}+A_{uv}}2=\frac12=P_{uu}=P_{vv}.
\]
Choose Gram vectors $p_v$ with $P_{uv}=\langle p_u,p_v\rangle$.  Equality in
Cauchy--Schwarz gives $p_u=p_v$ on every edge.  Since $G$ is connected, all
$p_v$ are equal, and hence
\[
   P=\frac12 J_n.
\]
For $u\ne v$, we now have
\[
   Q_{uv}=P_{uv}-A_{uv}
   =\begin{cases}
      \frac12,&uv\notin E(G),\\
      -\frac12,&uv\in E(G).
    \end{cases}
\]
Choose Gram vectors $q_v$ for $Q$. Since
$\|q_v\|^2=Q_{vv}=1/2$ for every $v$, and since for distinct $u,v$ we have
$Q_{uv}\in\{1/2,-1/2\}$, equality holds in the Cauchy--Schwarz inequality for
every pair $q_u,q_v$. More explicitly,
\[
   Q_{uv}=\frac12 \Longrightarrow \|q_u-q_v\|^2=0,
   \qquad
   Q_{uv}=-\frac12 \Longrightarrow \|q_u+q_v\|^2=0.
\]
Thus $q_u=q_v$ when $u$ and $v$ are nonadjacent, while $q_u=-q_v$ when they
are adjacent. Hence the vertices split into two sign classes, with no edges
inside either class and all possible edges between the two classes. Therefore
$G$ is complete bipartite.  Finally, its independence number is the size of its
larger part, so $\alpha(G)=n/2$ forces the two parts to have equal size.
\end{proof}

\section{The strict case \texorpdfstring{$\alpha<n/2$}{alpha < n/2}}

Throughout this section, $G$ is connected, nontrivial, satisfies equality in
\eqref{eq:KP-main}, and
\[
   \alpha:=\alpha(G)<\frac n2.
\]
Retain the notation $d=(n-\alpha)/n$ and $e=\alpha/n$, and set
\begin{equation}\label{eq:Rdef}
   R:=\frac{n}{\alpha}=\frac1e>2.
\end{equation}
Choose Gram representations
$
   P_{uv}=\langle p_u,p_v\rangle,
   \quad
   Q_{uv}=\langle q_u,q_v\rangle,
$
so that $\|p_v\|^2=\|q_v\|^2=d$ for every $v$.

\subsection{The two possible edge types}

\begin{lemma}[The two edge types]\label{lem:two-edge-types}
For every edge $uv\in E(G)$, exactly one of the following alternatives holds:
\begin{enumerate}
\item[\textup{(P)}]
\[
   B_{uv}=d-e,
   \qquad
   P_{uv}=d,
   \qquad
   Q_{uv}=-e,
   \qquad
   p_u=p_v;
\]
\item[\textup{(Q)}]
\[
   B_{uv}=e-d,
   \qquad
   P_{uv}=e,
   \qquad
   Q_{uv}=-d,
   \qquad
   q_u=-q_v.
\]
\end{enumerate}
\end{lemma}

\begin{proof}
Let $uv\in E(G)$. Since $G$ is simple, $A_{uv}=1$. Moreover,
\[
   P=\frac{B+A}{2},
   \qquad
   Q=\frac{B-A}{2},
   \qquad
   d+e=1.
\]
By \eqref{eq:Bedge}, either $B_{uv}=d-e$ or $B_{uv}=e-d$.
Because $\alpha<n/2$, we have $d>e$, so these two values are distinct.

Suppose first that $B_{uv}=d-e$. Then
\[
\begin{aligned}
   P_{uv}
   &=\frac{B_{uv}+A_{uv}}{2}
     =\frac{d-e+1}{2}
     =\frac{d-e+d+e}{2}
     =d,\\
   Q_{uv}
   &=\frac{B_{uv}-A_{uv}}{2}
     =\frac{d-e-1}{2}
     =\frac{d-e-d-e}{2}
     =-e.
\end{aligned}
\]
Since $P$ is the Gram matrix of the vectors $\{p_x:x\in V(G)\}$,
we have $P_{xy}=\langle p_x,p_y\rangle$ for all $x,y$. In addition,
$P_{uu}=P_{vv}=d$. Hence
\[
\begin{aligned}
   \|p_u-p_v\|^2
   &=\|p_u\|^2+\|p_v\|^2-2\langle p_u,p_v\rangle\\
   &=P_{uu}+P_{vv}-2P_{uv}\\
   &=d+d-2d=0.
\end{aligned}
\]
Therefore $p_u=p_v$, and alternative \textup{(P)} holds.

Suppose now that $B_{uv}=e-d$. Then
\[
\begin{aligned}
   P_{uv}
   &=\frac{B_{uv}+A_{uv}}{2}
     =\frac{e-d+1}{2}
     =\frac{e-d+d+e}{2}
     =e,\\
   Q_{uv}
   &=\frac{B_{uv}-A_{uv}}{2}
     =\frac{e-d-1}{2}
     =\frac{e-d-d-e}{2}
     =-d.
\end{aligned}
\]
Since $Q$ is the Gram matrix of the vectors $\{q_x:x\in V(G)\}$,
we have $Q_{xy}=\langle q_x,q_y\rangle$ for all $x,y$, and
$Q_{uu}=Q_{vv}=d$. Consequently,
\[
\begin{aligned}
   \|q_u+q_v\|^2
   &=\|q_u\|^2+\|q_v\|^2+2\langle q_u,q_v\rangle\\
   &=Q_{uu}+Q_{vv}+2Q_{uv}\\
   &=d+d+2(-d)=0.
\end{aligned}
\]
Thus $q_u=-q_v$, and alternative \textup{(Q)} holds.
\end{proof}

\subsection{\texorpdfstring{$P$-classes}{P-classes} and cells}

\begin{definition}\label{def:Pclasses-cells}
Two vertices $u,v$ are in the same \emph{$P$-class} if $p_u=p_v$.  Inside each
$P$-class, two vertices are in the same \emph{cell} if, in addition,
$q_u=q_v$.
\end{definition}

\begin{lemma}[Structure inside a $P$-class]\label{lem:inside-Pclass}
Let $\mathcal{C}$ be a $P$-class and let
$C_1,\ldots,C_r$ be its cells.  Then:
\begin{enumerate}[label=\textup{(\roman*)}]
\item each $C_i$ is independent;
\item every two distinct cells $C_i,C_j$ are completely joined;
\item if $x_i$ is the common $q$-vector of $C_i$, then
\begin{equation}\label{eq:simplex-gram}
   \langle x_i,x_i\rangle=d,
   \qquad
   \langle x_i,x_j\rangle=-e
   \quad(i\ne j);
\end{equation}
\item $r\le R$.
\end{enumerate}
\end{lemma}

\begin{proof}
If $u,v$ lie in the same cell, then $P_{uv}=Q_{uv}=d$, so
$A_{uv}=0$; hence each cell is independent.

Now take $u\in C_i$ and $v\in C_j$ with $i\ne j$.  Since $p_u=p_v$, we have
$P_{uv}=d$.  If $uv$ were a nonedge, then $Q_{uv}=P_{uv}=d$, which would imply
$q_u=q_v$, contrary to $i\ne j$.  Hence $uv$ is an edge.  By
Lemma~\ref{lem:two-edge-types}, it is an edge of type (P), and therefore
$Q_{uv}=-e$.  This proves complete adjacency and \eqref{eq:simplex-gram}.

The Gram matrix of $x_1,\ldots,x_r$ is
\[
   (d+e)I_r-eJ_r=I_r-eJ_r.
\]
Its eigenvalues are $1$ with multiplicity $r-1$ and $1-er$ with multiplicity
one.  Positive semidefiniteness gives $1-er\ge0$, that is,
$r\le 1/e=R$.
\end{proof}

\begin{lemma}[Edges between distinct $P$-classes]\label{lem:cross-edges}
Let $u,v$ lie in distinct $P$-classes.  Then
\begin{equation}\label{eq:cross-antipodal}
   uv\in E(G)
   \quad\Longleftrightarrow\quad
   q_v=-q_u.
\end{equation}
Moreover, a cell has at most one antipodal cell in the whole graph.
\end{lemma}

\begin{proof}
If $uv$ is an edge, it cannot be of type (P), because $p_u\ne p_v$.
Therefore it is of type (Q), and $q_v=-q_u$.

Conversely, suppose that $q_v=-q_u$ and that $uv$ is a nonedge.  Then
$P_{uv}=Q_{uv}=-d$.  Since $\|p_u\|^2=\|p_v\|^2=d$, equality in
Cauchy--Schwarz gives $p_v=-p_u$.  Hence, for every $w\in V(G)$,
\[
 A_{vw}
 =\langle p_v,p_w\rangle-\langle q_v,q_w\rangle
 =-\langle p_u,p_w\rangle+\langle q_u,q_w\rangle
 =-A_{uw}.
\]
All entries of $A$ are nonnegative, so the $u$- and $v$-rows of $A$ are both
zero.  This contradicts connectedness of the nontrivial graph $G$.  Thus
$uv$ is an edge, proving \eqref{eq:cross-antipodal}.

For uniqueness, suppose that two cells $D$ and $D'$ both have $q$-vector
$-x$, where $x$ is the $q$-vector of a cell $C$.  Choose $y\in D$ and
$z\in D'$.  Since $q_y=q_z$, the pair $yz$ cannot be an edge of type (Q).
If it were an edge of type (P), then $p_y=p_z$; together with $q_y=q_z$,
this would put $y$ and $z$ in the same cell.  Thus, if $D\ne D'$, the pair
$yz$ is a nonedge.  It follows that
\[
   P_{yz}=Q_{yz}=d.
\]
Since $P_{yy}=P_{zz}=d$, we obtain
\[
   \|p_y-p_z\|^2=P_{yy}+P_{zz}-2P_{yz}=0.
\]
Hence $p_y=p_z$, and therefore $D=D'$, a contradiction.  Thus the antipodal
cell is unique.
\end{proof}

\subsection{The weighted bipartite incidence graph}

The next construction translates the cell adjacencies into a weighted matching problem.

\begin{definition}[Cell incidence graph]\label{def:incidence-graph}
Let \(\mathcal C_1,\ldots,\mathcal C_s\) be the \(P\)-classes of \(G\).
We construct a weighted bipartite graph $\mathcal B$ with bipartition
$(L,Z)$ as follows.
The left part is
\[
L=\{u_1,\ldots,u_s\},
\]
where the vertex \(u_i\) represents the \(P\)-class \(\mathcal C_i\).  The
right part $Z$ consists of the vertices introduced below.

Each cell of \(G\) corresponds to exactly one edge of \(\mathcal B\).
The left endpoint of this edge is the vertex representing the
\(P\)-class containing the cell. Its right endpoint is determined as
follows.

Suppose that two cells \(C\in\mathcal C_i\) and
\(D\in\mathcal C_j\) have antipodal \(q\)-vectors. By
Lemma~\ref{lem:cross-edges}, the antipodal partner of a cell is unique,
and \(i\ne j\). Introduce one right vertex \(z_{\{C,D\}}\), and add the
two edges
\[
u_i z_{\{C,D\}}
\qquad\text{and}\qquad
u_j z_{\{C,D\}}.
\]
The first edge represents the cell \(C\), and the second edge represents
the cell \(D\). Assign them the weights
\[
w\bigl(u_i z_{\{C,D\}}\bigr)=|C|,
\qquad
w\bigl(u_j z_{\{C,D\}}\bigr)=|D|.
\]

If a cell \(C\in\mathcal C_i\) has no antipodal partner, introduce a new
right vertex \(z_C\), used only for \(C\), and add the single edge $u_i z_C$ with weight $w(u_i z_C)=|C|$.

Thus, the edges of \(\mathcal B\) are in one-to-one correspondence with
the cells of \(G\), and the weight of an edge is the cardinality of the
corresponding cell. A right vertex has degree \(2\) when it represents
an antipodal pair of cells, and degree \(1\) when it represents a cell
without an antipodal partner.
\end{definition}

\paragraph{Example of the cell incidence graph.}
Suppose that the cells of \(G\) are divided into two \(P\)-classes
\[
   \mathcal C_1=\{C_1,C_2\},
   \qquad
   \mathcal C_2=\{D_1,D_2\},
\]
where \(C_1\) and \(D_1\) have antipodal \(q\)-vectors, while \(C_2\) and
\(D_2\) have no antipodal partners.  Take
\[
   |C_1|=2,\qquad |C_2|=1,\qquad
   |D_1|=3,\qquad |D_2|=2.
\]
The paired cells \(C_1,D_1\) are represented by two edges sharing one right
endpoint, whereas each unpaired cell has a private right endpoint.  The
resulting weighted incidence graph is shown in
Figure~\ref{fig:cell-incidence-example}.

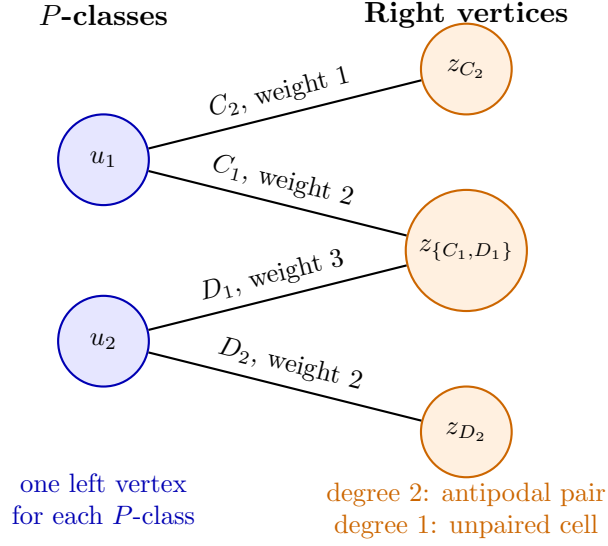
\begin{figure}[ht]
\centering
\begin{tikzpicture}[
    scale=.8,
    leftvertex/.style={
        circle,
        draw=blue!70!black,
        fill=blue!10,
        thick,
        minimum size=12mm
    },
    rightvertex/.style={
        circle,
        draw=orange!80!black,
        fill=orange!12,
        thick,
        minimum size=12mm
    },
    celledge/.style={
        thick
    },
    every node/.style={
        font=\small
    }
]

\node[font=\bfseries] at (0,3.9) {\(P\)-classes};
\node[font=\bfseries] at (6,3.9) {Right vertices};

\node[leftvertex] (u1) at (0,1.5) {\(u_1\)};
\node[leftvertex] (u2) at (0,-1.5) {\(u_2\)};

\node[rightvertex] (privateC2) at (6,3) {\(z_{C_2}\)};
\node[rightvertex] (pair) at (6,0) {\(z_{\{C_1,D_1\}}\)};
\node[rightvertex] (privateD2) at (6,-3) {\(z_{D_2}\)};

\draw[celledge]
    (u1) --
    node[above, sloped] {\(C_2\), weight \(1\)}
    (privateC2);

\draw[celledge]
    (u1) --
    node[above, sloped] {\(C_1\), weight \(2\)}
    (pair);

\draw[celledge]
    (u2) --
    node[above, sloped] {\(D_1\), weight \(3\)}
    (pair);

\draw[celledge]
    (u2) --
    node[above, sloped] {\(D_2\), weight \(2\)}
    (privateD2);

\node[align=center, blue!70!black] at (0,-4.1)
    {one left vertex\\for each \(P\)-class};

\node[align=center, orange!80!black] at (6,-4.3)
    {degree \(2\): antipodal pair\\
     degree \(1\): unpaired cell};

\end{tikzpicture}
\caption{A cell incidence graph.  Each cell is represented by one weighted
edge; antipodal cells correspond to edges sharing a right endpoint.}
\label{fig:cell-incidence-example}
\end{figure}

\subsection{Weighted matchings and their fractional relaxation}

We briefly recall the weighted matching optimization framework that will be
used below.  Let \(H\) be a graph and let
\[
   w:E(H)\longrightarrow \mathbb{R}_{>0}
\]
be a positive edge-weight function, and write $w_e:=w(e)$.  The
maximum-weight matching number of \((H,w)\) is
\[
   \nu_w(H)
   :=
   \max\left\{
      \sum_{e\in M}w_e:
      M\text{ is a matching in }H
   \right\}.
\]

A fractional matching of \(H\) is a vector
\(x=(x_e)_{e\in E(H)}\) satisfying
\[
   x_e\ge 0
   \qquad(e\in E(H))
\]
and
\[
   \sum_{e\ni v}x_e\le 1
   \qquad(v\in V(H)).
\]
Its weight is \(\sum_{e\in E(H)}w_ex_e\).  We denote the maximum possible
weight of a fractional matching by
\[
   \nu_w^*(H)
   :=
   \max\left\{
      \sum_{e\in E(H)}w_ex_e:
      \sum_{e\ni v}x_e\le1\ (v\in V(H)),\
      x_e\ge0\ (e\in E(H))
   \right\}.
\]

Equivalently, \(\nu_w^*(H)\) is the optimal value of the linear program
\begin{equation}\label{eq:fractional-matching-primal}
\begin{aligned}
\text{maximize}\quad
   &\sum_{e\in E(H)}w_ex_e,\\
\text{subject to}\quad
   &\sum_{e\ni v}x_e\le1
      &&(v\in V(H)),\\
   &x_e\ge0
      &&(e\in E(H)).
\end{aligned}
\end{equation}
Its dual program is
\begin{equation}\label{eq:fractional-matching-dual}
\begin{aligned}
\text{minimize}\quad
   &\sum_{v\in V(H)}y_v,\\
\text{subject to}\quad
   &y_u+y_v\ge w_{uv}
      &&(uv\in E(H)),\\
   &y_v\ge0
      &&(v\in V(H)).
\end{aligned}
\end{equation}

By strong linear-programming duality, the programs
\eqref{eq:fractional-matching-primal} and
\eqref{eq:fractional-matching-dual} have the same optimal value.

We shall use the following standard facts.

\begin{proposition}[Optimality relations for weighted bipartite matchings]
\label{prop:bipartite-matching-LP}
Let $H$ be a bipartite graph with positive edge weights $w$. Then
\[
   \nu_w(H)=\nu_w^*(H).
\]
Moreover, suppose that $x$ is feasible for
\eqref{eq:fractional-matching-primal} and $y$ is feasible for
\eqref{eq:fractional-matching-dual}, and that the two objective values are
equal. Then both $x$ and $y$ are optimal, and
\begin{align}
x_{uv}>0
   &\quad\Longrightarrow\quad
   y_u+y_v=w_{uv},
   \label{eq:optimality-edge}\\
y_v>0
   &\quad\Longrightarrow\quad
   \sum_{e\ni v}x_e=1.
   \label{eq:optimality-vertex}
\end{align}
\end{proposition}

\begin{proof}
The equality $\nu_w(H)=\nu_w^*(H)$ is the standard weighted matching theorem
for bipartite graphs; see, for example, \cite{LovaszPlummer}.

For the remaining statements, feasibility gives
\begin{align*}
&\sum_{v\in V(H)}y_v-\sum_{e\in E(H)}w_ex_e\\
&\qquad=
\sum_{v\in V(H)}
 y_v\left(1-\sum_{e\ni v}x_e\right)
+
\sum_{uv\in E(H)}
 x_{uv}\bigl(y_u+y_v-w_{uv}\bigr).
\end{align*}
Every term on the right-hand side is nonnegative. If the two objective values
are equal, the left-hand side is zero, so every term on the right must vanish.
It follows that a positive $x_{uv}$ makes the corresponding dual edge
inequality an equality, which gives \eqref{eq:optimality-edge}; similarly, a
positive $y_v$ makes the corresponding fractional-matching constraint an
equality, which gives \eqref{eq:optimality-vertex}. Equality of the objective
values also shows directly that both feasible solutions are optimal.
\end{proof}

\begin{lemma}[Cell adjacencies and matchings]
\label{lem:line-graph}
Two distinct cells are completely joined in $G$ if and only if the
corresponding edges of $\mathcal{B}$ share an endpoint.  Consequently,
\begin{equation}\label{eq:alpha-matching}
   \alpha(G)=\nu_w(\mathcal{B}),
\end{equation}
where $\nu_w(\mathcal{B})$ is the maximum weight of a matching in
$\mathcal{B}$.
\end{lemma}
\begin{proof}
Cells in the same $P$-class are completely joined by
Lemma~\ref{lem:inside-Pclass}; their edges in $\mathcal{B}$ share the
corresponding left endpoint.  Cells in different $P$-classes are completely
joined exactly when their $q$-vectors are antipodal, by
Lemma~\ref{lem:cross-edges}; their edges in $\mathcal{B}$ then share the
corresponding right endpoint.  There are no other cell adjacencies.

All vertices in a cell have the same $p$-vector and the same $q$-vector, and
therefore the same adjacency to every vertex outside the cell.  Thus any
independent set that meets a cell may be enlarged to contain the whole cell.
It follows that a maximum independent set is a union of pairwise nonadjacent
cells.  Under the preceding correspondence, such a collection is exactly a
matching in $\mathcal{B}$, and its cardinality is the sum of the cell weights.
This proves \eqref{eq:alpha-matching}.
\end{proof}

\begin{lemma}[Uniform fractional matching]
\label{lem:matching-LP}
Every \(P\)-class contains exactly \(R\) cells, all of the same cardinality.
In particular, \(R=n/\alpha\) is an integer.
\end{lemma}

\begin{proof}
For every edge \(f\in E(\mathcal B)\), set
\[
   x_f:=\frac1R.
\]
At a left vertex \(u_i\) corresponding to a \(P\)-class containing \(r_i\)
cells, Lemma~\ref{lem:inside-Pclass} gives
\[
   \sum_{f\ni u_i}x_f=\frac{r_i}{R}\le1.
\]
Every right vertex has degree at most two, and \(R>2\), so
\[
   \sum_{f\ni z}x_f\le\frac2R<1
\]
for every right vertex \(z\).  Thus \(x\) is a feasible fractional matching
of \(\mathcal B\).

Since the edges of \(\mathcal B\) correspond bijectively to the cells of
\(G\),
\[
   \sum_{f\in E(\mathcal B)}w_fx_f
   =
   \frac1R\sum_{f\in E(\mathcal B)}w_f
   =
   \frac nR
   =
   \alpha.
\]
On the other hand, Lemma~\ref{lem:line-graph} gives
\[
   \nu_w(\mathcal B)=\alpha.
\]
Because \(\mathcal B\) is bipartite, Proposition~\ref{prop:bipartite-matching-LP}
implies
\[
   \nu_w^*(\mathcal B)=\nu_w(\mathcal B)=\alpha.
\]
Hence \(x\) is an optimal fractional matching.

Let \(y\) be an optimal solution of the dual program
\eqref{eq:fractional-matching-dual}.  Since \(x_f=1/R>0\) for every edge
\(f=uv\), relation \eqref{eq:optimality-edge} gives
\[
   y_u+y_v=w_f
   \qquad(f=uv\in E(\mathcal B)).
\]
Every right-vertex constraint of the fractional matching \(x\) is strict.
Consequently, the contrapositive of \eqref{eq:optimality-vertex} gives
\[
   y_z=0
\]
at every right vertex \(z\).

It follows that, for every edge \(f=u_i z\),
\[
   w_f=y_{u_i}.
\]
Thus all edges incident with a fixed left vertex \(u_i\) have the same
weight.  Equivalently, all cells in a fixed \(P\)-class have the same
cardinality.

This common cardinality is positive, so \(y_{u_i}>0\).  Relation
\eqref{eq:optimality-vertex} therefore shows that the constraint at \(u_i\)
is an equality:
\[
   1=\sum_{f\ni u_i}x_f=\frac{r_i}{R}.
\]
Hence \(r_i=R\) for every \(P\)-class.  Since \(r_i\) is an integer, \(R\)
is an integer.
\end{proof}

\subsection{Cell-vector sums and the number of \texorpdfstring{$P$-classes}{P-classes}}

The cells from two $P$-classes are said to be paired by a perfect antipodal
matching if they can be paired bijectively so that each cell is matched with
the unique cell in the other $P$-class whose $q$-vector is its negative.

\begin{lemma}[Perfect matching between adjacent $P$-classes]
\label{lem:perfect-class-matching}
If two $P$-classes have at least one edge between them, then their cells are
paired by a perfect antipodal matching.
\end{lemma}
\begin{proof}
Let the common $q$-vectors of the cells in the first $P$-class be
$x_1,\ldots,x_R$.  By Lemmas~\ref{lem:inside-Pclass} and~\ref{lem:matching-LP}, the Gram matrix of \(x_1,\ldots,x_R\) is
\[
\Gamma
=
\bigl(\langle x_i,x_j\rangle\bigr)_{i,j=1}^R
=
I_R-\frac1R J_R.
\]
Let \(\mathbf 1_R\) denote the all-ones vector in \(\mathbb R^R\).
Then
\[
\begin{aligned}
\left\|x_1+\cdots+x_R\right\|^2
&=
\sum_{i=1}^R\sum_{j=1}^R
   \langle x_i,x_j\rangle\\
&=
\mathbf 1_R^{\mathsf T}\Gamma\mathbf 1_R\\
&=
\mathbf 1_R^{\mathsf T}
\left(I_R-\frac1R J_R\right)\mathbf 1_R\\
&=
R-\frac1R R^2\\
&=0.
\end{aligned}
\]
Consequently,
\begin{equation}\label{eq:sumzero}
x_1+\cdots+x_R=0.
\end{equation}

Suppose the second $P$-class has a cell with $q$-vector $y$.  Since the two
$P$-classes are adjacent, some cross-edge is present.  By Lemma~\ref{lem:two-edge-types},
the common inner product of their $p$-vectors is $e$.  Thus, for every cross
pair that is a nonedge, the corresponding $q$-inner product is also $e$.

If $y$ were not antipodal to any $x_i$, then
Lemma~\ref{lem:cross-edges} would imply that every cross-pair involving $y$ is
a nonedge.  Hence
\[
   \langle x_i,y\rangle=e
   \qquad(i=1,\ldots,R).
\]
Taking the inner product of \eqref{eq:sumzero} with $y$ gives
\[
   0=Re=1,
\]
a contradiction.  Thus every cell of the second $P$-class is antipodal to a
cell of the first.  Both classes have exactly $R$ cells, and antipodal partners
are unique by Lemma~\ref{lem:cross-edges}; therefore the matching is perfect.
\end{proof}

\begin{proposition}[At most two $P$-classes]\label{prop:at-most-two}
The graph $G$ has either one $P$-class or two $P$-classes.
\end{proposition}

\begin{proof}
Form an auxiliary graph whose vertices are the $P$-classes, with two classes
adjacent when there is an edge of $G$ between them.  By
Lemma~\ref{lem:perfect-class-matching}, if one $P$-class is adjacent to another,
then every one of its cells already has its unique antipodal partner in that
other class.  Lemma~\ref{lem:cross-edges} then prevents adjacency to any third
$P$-class.  Hence the auxiliary graph has maximum degree at most one.

Because $G$ is connected, the auxiliary graph is connected.  A connected graph
of maximum degree at most one has one or two vertices.  Thus $G$ has one or two
$P$-classes.
\end{proof}

\section{Identification of the connected equality graphs}

\begin{proposition}\label{prop:identify-strict}
Assume $\alpha(G)<n/2$.  Then either
\begin{enumerate}[label=\textup{(\roman*)}]
\item $G\cong K_{a,a,\ldots,a}$ with $R=n/\alpha$ parts, or
\item $G\cong H_R(a,b)$ for some positive integers $a,b$.
\end{enumerate}
\end{proposition}

\begin{proof}
By Proposition~\ref{prop:at-most-two}, there are one or two $P$-classes.

If there is one $P$-class, Lemma~\ref{lem:matching-LP} says that it contains
exactly $R$ cells, all of the same size, say $a$.  By
Lemma~\ref{lem:inside-Pclass}, each cell is independent and every two distinct
cells are completely joined.  Thus
\[
   G\cong K_{a,a,\ldots,a}
\]
with $R$ parts.

If there are two $P$-classes, each contains exactly $R$ equal-sized cells.
Let their respective cell sizes be $a$ and $b$.  Each $P$-class induces a
balanced complete $R$-partite graph.  By
Lemma~\ref{lem:perfect-class-matching}, their cells are paired bijectively;
each paired pair of cells is completely joined, and there are no other
cross-edges.  This is exactly the graph $H_R(a,b)$ of
Definition~\ref{def:Hrab}.
\end{proof}

\section{Spectra of the extremal families}

We now verify that all graphs listed in Theorem~\ref{thm:main} attain equality.
For balanced complete multipartite graphs, we use their standard adjacency
spectrum; for the family $H_r(a,b)$, we give a direct calculation.

\begin{lemma}[Balanced complete multipartite graphs]
\label{lem:spectrum-balanced-multipartite}
For $r\ge2$ and $a\ge1$, using a superscript $[t]$ to denote
multiplicity $t$,
\[
 \Spec\bigl(K_{a,a,\ldots,a}\bigr)
 =\bigl\{a(r-1),\,(-a)^{[r-1]},\,0^{[r(a-1)]}\bigr\}.
\]
Consequently,
\[
   \E\bigl(K_{a,a,\ldots,a}\bigr)=2a(r-1)
   =2\bigl(ra-a\bigr).
\]
\end{lemma}

\begin{proof}
The displayed spectrum is standard; see, for example,
\cite{BrouwerHaemers}.  Since the graph has order $ra$ and independence
number $a$, the energy formula gives
\[
   \E\bigl(K_{a,a,\ldots,a}\bigr)
   =2a(r-1)
   =2\bigl(|V(K_{a,a,\ldots,a})|-\alpha(K_{a,a,\ldots,a})\bigr).
\]
\end{proof}

\begin{lemma}[Spectrum of $H_r(a,b)$]\label{lem:spectrum-H}
Let $r\ge3$ and $a,b\ge1$.  Define
\[
 \mu_{\pm}
 :=\frac{(r-1)(a+b)
 \pm\sqrt{(r-1)^2(a-b)^2+4ab}}{2}.
\]
Then $\mu_+>\mu_->0$ and
\begin{equation}\label{eq:spectrum-H}
 \Spec\bigl(H_r(a,b)\bigr)
 =\bigl\{\mu_+,\mu_-,\,(-(a+b))^{[r-1]},\,
 0^{[r(a+b)-r-1]}\bigr\}.
\end{equation}
Moreover,
\[
   \alpha\bigl(H_r(a,b)\bigr)=a+b
\]
and
\[
   \E\bigl(H_r(a,b)\bigr)
   =2(r-1)(a+b)
   =2\bigl(r(a+b)-(a+b)\bigr).
\]
\end{lemma}

\begin{proof}
For any distinct indices $i$ and $j$, the set $X_i\cup Y_j$ is
independent and has cardinality $a+b$.  On the other hand, since
distinct $X$-cells are mutually complete and distinct $Y$-cells are
mutually complete, an independent set can meet at most one $X$-cell
and at most one $Y$-cell.  Moreover, it cannot meet both $X_i$ and
$Y_i$.  Therefore
\[
   \alpha(H_r(a,b))=a+b.
\]

We next determine the adjacency spectrum.  Consider the equitable
partition
\[
   X_1,\ldots,X_r,Y_1,\ldots,Y_r.
\]
Its quotient matrix is
\[
 Q=
 \begin{pmatrix}
   a(J_r-I_r) & bI_r\\
   aI_r & b(J_r-I_r)
 \end{pmatrix}.
\]
Let
\[
   D=\operatorname{diag}(aI_r,bI_r).
\]
Although $Q$ need not be symmetric, it is similar to the symmetric
normalized quotient matrix
\[
 M=D^{1/2}QD^{-1/2}
 =
 \begin{pmatrix}
   a(J_r-I_r) & \sqrt{ab}\,I_r\\
   \sqrt{ab}\,I_r & b(J_r-I_r)
 \end{pmatrix}.
\]
Hence $Q$ and $M$ have the same eigenvalues.

The orthogonal complement of the cell-constant subspace consists of
vectors whose coordinates sum to zero on every cell.  Since every cell
is independent and all vertices in a fixed cell have identical
neighborhoods outside that cell, the adjacency matrix annihilates this
subspace.  Its dimension is
\[
   r(a-1)+r(b-1)=r(a+b-2),
\]
so it contributes $r(a+b-2)$ zero eigenvalues.  It remains to determine
the eigenvalues of $M$.

Decompose
\[
   \mathbb R^r=\operatorname{span}\{\one\}\oplus\one^\perp,
\]
where $\one$ is the all-ones vector in $\mathbb R^r$.  Since
\[
   (J_r-I_r)\one=(r-1)\one
   \quad\text{and}\quad
   (J_r-I_r)z=-z
   \quad(z\in\one^\perp),
\]
the two summands can be treated separately.

On the two-dimensional subspace generated by $(\one,0)$ and
$(0,\one)$, the matrix $M$ is represented by
\[
 B_0=
 \begin{pmatrix}
   a(r-1) & \sqrt{ab}\\
   \sqrt{ab} & b(r-1)
 \end{pmatrix}.
\]
Its eigenvalues are
\[
 \mu_\pm=
 \frac{(r-1)(a+b)\pm
 \sqrt{(r-1)^2(a-b)^2+4ab}}{2}.
\]
Moreover,
\[
   \operatorname{tr}(B_0)=(r-1)(a+b)>0
\]
and, since $r\ge 3$,
\[
   \det(B_0)
   =ab\bigl((r-1)^2-1\bigr)
   =abr(r-2)>0.
\]
Thus both $\mu_+$ and $\mu_-$ are positive.

Now let $z\in\one^\perp$.  Since $J_rz=0$, we have
$(J_r-I_r)z=-z$.
Therefore the subspace generated by $(z,0)$ and $(0,z)$ is invariant
under $M$, and the restriction of $M$ to this subspace is represented by
\[
 B_1=
 \begin{pmatrix}
   -a & \sqrt{ab}\\
   \sqrt{ab} & -b
 \end{pmatrix}.
\]
Its characteristic polynomial is
\[
   \det(\lambda I_2-B_1)
   =(\lambda+a)(\lambda+b)-ab
   =\lambda(\lambda+a+b).
\]
Hence its eigenvalues are $0$ and $-(a+b)$.  Since
$\dim\one^\perp=r-1$, each of these eigenvalues has multiplicity $r-1$.

Combining the eigenvalues of $M$ with those arising from the
orthogonal complement of the cell-constant subspace gives
\[
 \Spec(H_r(a,b))
 =
 \left\{
   \mu_+,\,
   \mu_-,\,
   \bigl(-(a+b)\bigr)^{[r-1]},\,
   0^{[r(a+b-1)-1]}
 \right\},
\]
which proves \eqref{eq:spectrum-H}.

Finally,
\[
   \mu_++\mu_-=(r-1)(a+b).
\]
Therefore
\begin{align*}
   \E(H_r(a,b))
   &=
   \mu_++\mu_-+(r-1)(a+b)\\
   &=2(r-1)(a+b).
\end{align*}
Since $|V(H_r(a,b))|=r(a+b)$ and
$\alpha(H_r(a,b))=a+b$, this may also be written as
\[
   \E(H_r(a,b))
   =2\bigl(|V(H_r(a,b))|-\alpha(H_r(a,b))\bigr).
\]
\end{proof}

\section{Proof of the main theorem}

\begin{proof}[Proof of Theorem~\ref{thm:main}]
First suppose that $G$ is connected.

If $G$ has one vertex, it is an isolated vertex and equality is immediate.  If
$G$ is nontrivial and $\alpha(G)=n/2$, Proposition~\ref{prop:half} gives
$G\cong K_{n/2,n/2}$, which is a balanced complete multipartite graph.  If
$\alpha(G)<n/2$, Proposition~\ref{prop:identify-strict} gives either a balanced
complete multipartite graph or a graph $H_r(a,b)$ with $r=R>2$, hence
$r\ge3$.  This proves necessity for connected graphs.

Conversely, an isolated vertex satisfies equality trivially.
Lemma~\ref{lem:spectrum-balanced-multipartite} proves equality for balanced
complete multipartite graphs, and Lemma~\ref{lem:spectrum-H} proves equality
for $H_r(a,b)$.

Finally, let
\[
   G=G_1\dot\cup\cdots\dot\cup G_k
\]
be the decomposition into connected components.  Both energy and independence
number are additive over components:
\[
   \E(G)=\sum_{i=1}^k\E(G_i),
   \qquad
   \alpha(G)=\sum_{i=1}^k\alpha(G_i).
\]
Thus, if every component is listed in the theorem, then
\[
 \E(G)
 =\sum_i2\bigl(|V(G_i)|-\alpha(G_i)\bigr)
 =2\bigl(n-\alpha(G)\bigr).
\]
Conversely, if equality holds for $G$, then
\[
 0
 =\E(G)-2(n-\alpha(G))
 =\sum_i\left[
   \E(G_i)-2\bigl(|V(G_i)|-\alpha(G_i)\bigr)
 \right].
\]
Every summand is nonnegative by \eqref{eq:KP-main}, so every summand is zero.
Hence every connected component is an equality graph and therefore belongs to
one of the listed families.
\end{proof}

\section{Further consequences}

\begin{corollary}\label{cor:integer-ratio}
If $G$ is a connected equality graph with $\alpha(G)<|V(G)|/2$, then
$|V(G)|/\alpha(G)$ is an integer.
\end{corollary}

\begin{proof}
This is contained in Lemma~\ref{lem:matching-LP}.
\end{proof}

\begin{corollary}\label{cor:bipartite}
A bipartite graph $G$ satisfies
\[
   \E(G)=2\bigl(|V(G)|-\alpha(G)\bigr)
\]
if and only if every nontrivial connected component of $G$ is a balanced
complete bipartite graph.
\end{corollary}

\begin{proof}
The graphs $H_r(a,b)$ with $r\ge3$ contain triangles inside each of their two
complete multipartite halves, and a balanced complete multipartite graph is
bipartite only when it has exactly two parts.  The result follows from
Theorem~\ref{thm:main}.
\end{proof}

\end{document}